\documentclass{article}

\usepackage[utf8]{inputenc}
\usepackage{fullpage}
\usepackage{mathtools}
\usepackage{amsfonts,amsmath,amssymb,amsthm}

\newtheorem{theorem}{Theorem}[section]
\newtheorem{definition}[theorem]{Definition}

\newtheorem{lemma}[theorem]{Lemma}
\newtheorem{example}[theorem]{Example}
\newtheorem{proposition}[theorem]{Proposition}
\newtheorem{corollary}[theorem]{Corollary}

\title{Coloring cross-intersecting families}
\author{Danila Cherkashin\footnote{Chebyshev Laboratory, St. Petersburg State University, 14th Line V.O., 29B, Saint Petersburg 199178 Russia; 
Moscow Institute of Physics and Technology, Lab of advanced combinatorics and network applications, Institutsky lane 9, Dolgoprudny, Moscow region, 141700, Russia; St.~Petersburg Department of V.~A.~Steklov Institute of Mathematics of the Russian Academy of Sciences.}}

\date{April 2017}

\begin{document}

\maketitle

\begin{abstract}
Intersecting and cross-intersecting families usually appear in extremal combinatorics in the vein of the Erd{\H o}s--Ko--Rado theorem~\cite{erdos1961intersection}.
On the other hand, P.~Erd{\H o}s and L.~Lov{\'a}sz in the noted paper~\cite{EL} posed problems on coloring intersecting families
as a restriction of classical hypergraph coloring problems to a special class of hypergraphs.
This note deals with the mentioned coloring problems stated for cross-intersecting families.
\end{abstract}

\section{Introduction}

Intersecting families in extremal combinatorics appeared in~\cite{erdos1961intersection}, and a large branch of extremal combinatorics starts from this paper.

\begin{definition}
Intersecting family is a hypergraph $H = (V, E)$ such that $e \cap f \neq \emptyset$ for every $e, f \in E$.
\end{definition}

Then P.~Erd\H{o}s and L.~Lov{\'a}sz in~\cite{EL} introduced several problems on coloring intersecting families (\textit{cliques} in the original notation), i.\,e.~hypergraphs without a pair of disjoint edges.
Obviously, an intersecting family could have chromatic number 2 or 3 only; the main interest refers to chromatic number 3. 
Unfortunately, there is no ``random'' example of such family, so the set of known intersecting families with chromatic number 3 is very poor.

Cross-intersecting families were introduced to study maximal and almost-maximal intersecting families 
(the notation appears in~\cite{matsumoto1989exact}).

\begin{definition}
Cross-intersecting family is a hypergraph $H = (V, E = A \cup B)$ such that every $a \in A$ intersects every $b \in B$, and $A$, $B$ are not empty.
\end{definition}

Also, the Hilton--Milner theorem~\cite{hilton1967some} and the Frankl theorem~\cite{frankl1987erdos} should be noted.
Recently a general approach to mentioned problems was introduced by A. Kupavskii and D. Zakharov~\cite{kupavskii2016regular} (the reader can also see this paper for a survey).

\subsection{The chromatic number}

We are interested in vertex colorings of cross-intersecting families. 
Coloring is \textit{proper} if there are no monochromatic edges.
\textit{Chromatic number} is the minimal number of colors that admits a proper coloring.
First, note that a cross-intersecting family could have an arbitrarily large chromatic number.

\begin{example}
Consider an arbitrary integer $r > 1$.
Consider a hypergraph $H_0 = (V_0, E_0)$ with chromatic number $r$.
Put $A := E_0$, $B := \{V_0\}$. Obviously, $H := (V_0, A, B)$ is a cross-intersecting family with chromatic number $r$.
\end{example}

\noindent However, under a natural assumption (note that it holds for any $n$-uniform hypergraph) a chromatic number of a cross-intersecting family is bounded.

\begin{proposition}
Let $H = (V, A, B)$ be a cross-intersecting family.
Suppose that $A$ and $B$ both have minimal elements of $E$, i.\,e.~there are such $a \in A$, $b \in B$ that $a$, $b$ both have no subedge in $H$.
Then $\chi(H) \leq 4$.
\end{proposition}

\begin{proof}
Let us color $a \cap b$ in color 1, $a \setminus b$ in color 2, $b \setminus a$ in color 3 and all other vertices in color 4.
One can see that the coloring is proper because both $a$ and $b$ have no subedge.
\end{proof}

It turns out, that if there is no pair $e_1$, $e_2 \in E$ such that $e_1 \subset e_2$ and every edge has a size of at least 3,
then the cross-intersecting family can have chromatic number 2 or 3 only. Moreover, the following theorem holds.

\begin{theorem}
Let $H = (V, A, B)$ be a cross-intersecting family such that there is no pair $e_1, e_2 \in A \cup B$ such that $e_1 \subset e_2$ (i.\ e. $(V, E)$ is a Sperner system). 
Then $\chi(H) \leq 3$ or $V := \{v_1, \dots, v_m, u_1, \dots u_l\}$; $B := \{ \{v_1, \dots, v_m\}, \{u_1, \dots u_l\} \}$; $A := \{ \{v_i, u_j\} \mbox{ for all } i, j \}$ (modulo $A$-$B$ symmetry), where $m$, $l \geq 2$.
\label{chi23}
\end{theorem}

\begin{corollary}
Let $H = (V, A, B)$ be an $n$-uniform cross-intersecting family. Then $\chi (H) \leq 3$ or $n=2$ and $H = K_4$.
\label{corchi23}
\end{corollary}

\begin{corollary}
Let $H = (V, A, B)$ be an $n$-uniform cross-intersecting family and $\min (|A|, |B|) \geq 3$. Then $\chi (H) \leq 3$.
\label{cor2} 
\end{corollary}

\subsection{Maximal number of edges}

It turns out that the maximal number of edges in a ``nontrivial'' $n$-uniform intersecting family is bounded.
There are two ways to formalize the notion ``nontrivial''. The first one is to say that $\chi(H) \geq 3$ (denote the corresponding maximum by $M(n)$). The second one says that $H$ is nontrivial if and only if $\tau (H) = n$ (denote the corresponding maximum by $r(n)$), where $\tau (H)$ is defined below.

\begin{definition}
Let $H = (V, E)$ be a hypergraph. The \textit{covering number} (also known as \textit{transversal number} or \textit{blocking number}) of $H$ is the smallest integer $\tau (H)$ that
there is a set $A \subset V$ such that every $e \in E$ intersects $A$ and $|A| = \tau$.
\end{definition}

\subsubsection{Upper bounds.} Obviously, $M(n) \leq r(n)$. P.~Erd\H{o}s and L.~Lov{\'a}sz proved in~\cite{EL} that $r(n) \leq n^{n}$ 
(one can find slightly better bound in~\cite{cherkashin2011maximal}).
The best current upper bound is $r(n) \leq cn^{n-1}$ (see~\cite{arman2017upper}). 
Surprisingly, we can prove a very similar statement for cross-intersecting families. Let us introduce a ``nontriviality'' notion for cross-intersecting families. 

\begin{definition}
Let us call a cross-intersecting family $H = (V, A, B)$ \textit{critical} if 
\begin{itemize}
    \item for any edge $a \in A$ and any $v \in a$ there is $b \in B$ such that $a \cap b = \{v\}$;
    \item for any edge $b \in B$ and any $v \in b$ there is $a \in A$ such that $a \cap b = \{v\}$.
\end{itemize}
\end{definition}

\noindent Note that if an $n$-uniform intersecting family $H = (V,E)$ has $\tau (H) = n$ then $(V, E, E)$ is a critical cross-intersecting family.

\begin{theorem}
\label{max}
Let $H = (V, A, B)$ be a critical cross-intersecting family. Denote
$$n: = \max_{e \in A \cup B} |e|.$$
Then 
$$\max(|A|, |B|) \leq n^n.$$ 
\end{theorem}

\subsubsection{Lower bounds.} L.~Lov{\'a}sz conjectured that $M(n) = [(e - 1)n!]$ (an example was constructed in~\cite{EL}).
This was disproved by P. Frankl, K. Ota and N. Tokushige~\cite{frankl1996covers}.
They have provided an explicit example of an $n$-uniform hypergraph $H$ with $\tau (H) = n$ and 
\begin{equation}
c \left (\frac{k}{2} + 1 \right)^{k-1}
\label{2}
\end{equation}
edges. For cross-intersecting families Example~\ref{nn} shows that Theorem~\ref{max} is tight.

\subsection{The set of the pairwise edge intersection sizes}

\begin{definition}
For a hypergraph $H = (V,E)$ let us consider the set of the sizes of pairwise edge intersections:
$$Q(H) := \{|e_1 \cap e_2|, e_1,e_2 \in E\}.$$
\end{definition}

\noindent Again, P.~Erd\H{o}s and L.~Lov{\'a}sz showed that for an $n$-uniform intersecting family $H$ one has
$3 \leq |Q(H)|$ for sufficiently large $n$, but there is no example with $|Q(H)| < \frac{n-1}{2}$. For cross-intersecting families there is a simple example with $|Q(H)| = 4$.

\begin{theorem}
\label{edgint}
There is an $n$-uniform cross-intersecting family $H$ with $Q(H) = \{0, 1, 2, n-1\}$ and $\chi (H) = 3$.
\end{theorem}

\noindent See Example~\ref{exex} for the proof.

\subsection{Examples}
Unlike the case of intersecting families there is a method of constructing a large set of (critical) cross-intersecting families with chromatic number 3, based 
on percolation.
This method makes it possible to construct a cross-intersecting family from a random planar triangulation.

\begin{example}
Consider an arbitrary planar triangulation with external face $F$ that has a size of at least 4.
Split $F$ into 4 disjoint connected parts $F_1$, $F_2$, $F_3$, $F_4$.
Let $A_0$ be the set of collections of vertices that form a simple path from $F_1$ to $F_3$; $B_0$ be the set of collections of vertices that form a simple path from $F_2$ to $F_4$.
Finally, let $A \subset A_0$, $B \subset B_0$ be the sets of all minimal (by the inclusion relation) subsets; $H = (V, E)$.

Obviously, $\chi (H) = 3$ (one may see that no example with chromatic number 4 could be obtained from a planar triangulation).
\end{example}

For a given $n > 2$ there exists an $n$-uniform cross-intersecting family (not critical) with chromatic number 3 and an arbitrarily large number of edges.

\begin{example}
Let $m$ be an arbitrary integer number.
Put $V(H) := \{v_1, \dots, v_{2n-1} \} \cup \{u_1, \dots, u_m \}$; 
$E(H) := A_1 \cup A_2 \cup B_1 \cup B_2$, where $A_1 \cup B_1$ is the set of all $n$-subsets of $\{v_1, \dots, v_{2n-1} \}$, 
$A_1$ contains edges intersecting $\{v_1, \dots v_{n-1}\}$, $B_1$ contains edges intersecting $\{v_1, v_n, \dots, v_{2n-3} \}$ 
(so $A_1 \cap B_1 \neq \emptyset$),
$$A_2 := \{ \{v_1, \dots v_{n-1}, u_i \} \ \text {for every} \ i\},$$
$$B_2 := \{ \{v_1, v_n, \dots, v_{2n-3}, u_i \} \  \text{for every} \ i\}.$$
Note that $H_1 := (V_1, A_1 \cup B_1)$ has chromatic number 3, so $\chi (H) \geq 3$, hence by Corollary~\ref{corchi23} we have $\chi (H) = 3$.

Let us show that $H$ is a cross-intersecting family. 
Clearly, since $A_1$, $B_1 \subset V_1$ every edge from $A_1$ intersects with every edge from $B_1$.
By the definition every edge of $A_2$ contains $\{v_1, \dots v_{n-1} \}$ so it intersects with every edge from $B_1$; by symmetry the same holds for $B_2$ and $A_1$.
Also every edge from $A_2$ intersects with every edge from $B_2$ at the point $v_1$. 
\end{example}

\begin{example}
\label{nn}
Consider an arbitrary $n > 1$.
Let $V := \{v_{ij}\ | \  1 \leq i, j \leq n \}$, $A := \{ \{v_{i1}, \dots v_{in} \} \ | \ 1 \leq i \leq n  \}$, 
$B := \{ \{v_{1i_1}, v_{2i_2}, \dots, \ v_{ni_n} \} \ | \ 1 \leq i_1, i_2, \dots, i_n \leq n \}$.
Note that $|A| = n$, $|B| = n^n$.
Obviously, $H := (V, A, B)$ is a cross-intersecting family and $\chi (H) = 3$.
\end{example}

\begin{example}[Proof of Theorem~\ref{edgint}]
\label{exex}
Our construction is based on the following object. 

\begin{definition}
A hypergraph is called simple if every two edges share at most one vertex.
\end{definition}

Let us take an $(n-1)$-uniform simple hypergraph $H_0 = (V_0, E_0)$ such that $\chi (H) = 3$ (see~\cite{EL, kostochka2010constructions} for constructions).
Denote $V := V_0 \sqcup \{u_1, \dots, u_n\}$, $B := \{\{u_1, \dots, u_n \}\}$, $A := \{e \cup \{u_i\} | e \in E_0, 1 \leq i \leq n\}$. By the construction, $H$ is an $n$-uniform cross-intersecting family.

Let us show that $\chi (H) = 3$. Suppose the contrary, i.e.~there is a 2-coloring of $V$ without monochromatic edges of $A \cup B$. 
By the definition of $H_0$ every 2-coloring of $V_0$ gives a monochromatic (say, blue) edge $e \in E_0$.
Then every $u_i$ is red, otherwise $e \cap \{u_i\}$ is monochromatic. So $\{u_1, \dots, u_n\}$ is red, a contradiction.

Note that $Q(H_0) = \{0,1\}$, so $Q(H) = \{0, 1, 2, n-1\}$.
\end{example}

\section{Proofs}

\begin{proof}[Proof of Theorem~\ref{chi23}]
First, suppose that there is no edge of size 2.
Consider such a pair $a \in A$, $b \in B$ that $|a \cup b|$ is the smallest.
Pick arbitrary vertices $v_a \in a \setminus b$ and $v_b \in b \setminus a$. Let us color $v_a$ and $v_b$ in color 1, $a \cup b \setminus \{v_a, v_b\}$ in color 2 and the remaining vertices in color 3.

Let us show that this coloring is proper. Since there is no edge of size 2, there is no edge of color 1.
Every edge intersects $a$ or $b$, so there is no edge of color 3.
Suppose that there is an edge $e$ of color 2. Without loss of generality $e \in A$.
Then $e \subset |a \cup b \setminus \{v_a\}|$, so $|e \cup b| < |a \cup b|$, a contradiction.

Now let us consider the case $\{u, v\} \in E(H)$. 

\begin{lemma}
Let $a = \{u, v\} \in A$, $u \in b \in B$. Then for every $w \in B$ there is the edge $\{v, w\} \in E(H)$ or $\chi (H) \leq 3$.
\end{lemma}

\begin{proof}
Suppose that $\chi (H) > 3$. Then for every $w \in b$ there is the edge $\{w, v\} \in E (H)$, 
otherwise one can color $v$, $w$ in color 1, $b \setminus w$ in color 2 and all other vertices in color 3, producing a proper 3-coloring.
\end{proof}

Without loss of generality $\{u, v\} \in A$.
Consider any edge $b \in B$ (without loss of generality $u \in b$). By lemma there is every edge $\{v, w\} \in E(H)$ for $w \in B$.
Suppose that for some $w \in b$ there is the edge $\{v, w\} \in B$. 
Then, by lemma (for $a = \{u, v\}$ and $b = \{v, w\}$) we have $\{u, w\} \in E(H)$, so $b = \{u, w\}$. 
So $H$ contains a triangle on $\{u, v, w\}$ with edges both in $A$ and $B$ ($\star$).
If $H$ coincides with the triangle on $\{u, v, w\}$, then $\chi (H) = 3$. Otherwise, $H$ contains $e$ which does not 
intersect one of the edges $\{u, v\}$, $\{u, w\}$, $\{v, w\}$.
So, we can change denotation as follows: $\{u, v, w\} = \{q, r, s\}$, such that $e$, $\{q, r\} \in B$ and $e \cap \{q,r\} = \emptyset$. 
Note that one of the edges $\{q, s\}$, $\{r, s\}$ lies in $A$ (without loss of generality it is $\{q, s\}$).

By lemma (for $a = \{q, s\}$ and $b = e$) there is an edge $\{q, t\}$ for every $t \in e$. If $\{r, s\} \in B$, then $\{q, t\} \in A$ for every $t \in e \setminus s$.
So by lemma (for $a = \{q, t\}$ and $b = \{q, r\}$) there is an edge $\{r, t\}$ for every $t \in e$. If $\{r, s\} \in A$, then by lemma again (for $a = \{r, s\}$ and $b = e$) there is an edge $\{r, t\}$ for every $t \in e$.
Summing up, we have edges $\{q, r\}$, $e \in B$,  $\{q, s\} \in A$ and $\{x, t\} \in E(H)$ for every choice $x \in \{q, r\}$ and $t \in e$.

Suppose that $|e| > 2$. It means that there are different $s, t_1, t_2 \in e$. 
Note that $\{r, t_1\} \in A$ since $\{q, s\} \in A$, so $\{q, t_2\} \in A$. Thus every edge $\{x, t\} \in A$ for every choice $x \in \{q, r\}$ and $t \in e$.
Obviously, we have listed all edges of the hypergraph, so we proved the claim in this case. 
Note also that the set of colors in $\{q, r\}$ does not intersect the set of colors in $e$, so $\chi (H) = 4$. 
If $|e| = 2$, then $H = K_4$, and again $\chi (H) = 4$.

In the remaining case we have all $\{v, w\}$ in $A$. If $|B| = 1$, then $\chi (H) \leq 3$, so there is an edge 
$b' \in B$, such that it does not contain $u$.
Suppose that $b \cap b' = \emptyset$. 
Then by lemma (for $a = \{v, w\}$ and $b'$) we have edges $\{w, t'\}$ for every $w \in b$ and $t' \in b'$. 
Obviously, all these edges lie in $A$, otherwise we are done by the first case 
(if some $\{w, t'\} \in B$, then we have $\{w, v\} \in A$, $\{w, t'\}$, $b' \in B$).

If $b \cap b' \neq \emptyset$, then by lemma for $a = \{u, v\}$ and $b$ we have an edge $\{v, t\}$ for some $t \in b \cap b'$. Then $b' = \{v, t\}$.
Analogously, $b = \{u, t\}$. So the condition ($\star$) holds, and we are done. 

\end{proof}

\begin{proof}[Proof of Theorem~\ref{max}]
First, we need the following definition.
\begin{definition}
Let $H = (V,E)$ be a hypergraph and $W$ be a subset of $V$. Define
$$H_W := (V \setminus W, \{e \setminus W \ |\ e \in E \}).$$
Then $H$ is a \text{flower} with $k$ \text{petals} with core $W$ if $\tau (F_W) \geq k$.
\end{definition}

\noindent The following Lemma was proved by J. H\r{a}stad, S. Jukna and P. Pudl{\'a}k~\cite{haastad1995top}. 
We provide its proof for the completeness of presentation.

\begin{lemma}
Let $H = (V, E)$ be a hypergraph; $n := \max_{e \in E} |e|$.
If $|E| > (k - 1)^n$ then $F$ contains a flower with $k$ petals.
\end{lemma}

\begin{proof} Induction on $n$. The basis $n = 1$ is trivial.

Now suppose that the lemma is true for
$n - 1$ and prove it for $n$. 
If $\tau (H) \geq k$ then $H$ itself is a flower
with at least $k$ petals (and an empty core). 
Otherwise, some set of size $k - 1$ intersects all the edges of $H$, and hence, at least 
$|E|/(k - 1)$
of the edges must contain some vertex $x$. 
The hypergraph
$H_{\{x\}} = (V_{\{x\}}, E_{\{x\}})$ 
has
$$|E_{\{x\}}| \geq \frac{|E|}{k - 1} > (k - 1)^{n-1}$$
edges, each of cardinality at most $n - 1$. 
By the induction hypothesis, 
$H_{\{x\}}$ contains a flower with $k$ petals and some core $Y$. 
Adding the element $x$ back to the sets in this flower, we obtain a flower in $H$ with the
same number of petals and the core $Y \cup \{x\}$. 
\end{proof}

Now let us prove Theorem~\ref{max}. Suppose the contrary, i. e. that, without loss of generality, $|A| \geq n^n + 1$.
Then by Lemma the hypergraph $(V, A)$ contains a flower with $n+1$ petals.
It means that every $b \in B$ intersects the core of the flower, and $H$ is not critical. A contradiction.
\end{proof}

\section{Open questions}

The most famous problem in hypergraph coloring is to determine the minimal number of edges in 
an $n$-uniform hypergraph with $\chi (H) = 3$ (it is usually denoted by $m(n)$). The best known bounds (\cite{Erdos2, radhakrishnan1998improved, cherkashin2015note}) are
\begin{equation}
c \sqrt {\frac{n}{\ln n}} 2^n \leq m(n) \leq \frac{e \cdot \ln 2 }{4} n^2  2^{n} (1 + o(1)).    
\label{1}
\end{equation}
P.~Erd\H{o}s and L.~Lov{\'a}sz in~\cite{EL} posed the same question for the class of intersecting families.
Even though the intersecting condition is very strong, it does not provide a better lower bound.
On the other hand, the upper bound in (\ref{1}) is probabilistic, so it does not work for intersecting families.
So the asymptotically best upper bound is $7^{\frac{n-1}{2}}$ for $n = 3^k$, which is given by iterated Fano plane.

Another question is to determine the minimal size $a(n)$ of the largest intersection in an $n$-uniform intersecting family.
The best bounds at this time are 
$$\frac{n}{\log_2 n} \leq a(n) \leq n-2.$$

\noindent Studying the mentioned problems for cross-intersecting families is also of interest.

Recall that Example~\ref{nn} shows that Theorem~\ref{max} is tight. 
On the other hand, $\max \min (|A|, |B|)$ over all cross-intersecting families with chromatic number 3 is unknown.
Obviously, one may take the example $(V, E)$ by P. Frankl, K. Ota and N. Tokushige and put $A = B = E$ to get lower bound (\ref{2}).

\subsubsection{Acknowledgements.} 
The work was supported by the Russian Scientific Foundation grant 16-11-10014. 
The author is grateful to A. Raigorodskii and F. Petrov for constant inspiration, to A. Kupavskii for historical review and
for directing his attention to the paper~\cite{haastad1995top} and to N. Rastegaev for very careful reading of the draft of the paper.

\bibliographystyle{plain}
\bibliography{main}

\end{document}